\documentclass[11pt]{article}

\usepackage{amssymb,amsmath,amsfonts,amsthm}
\usepackage{latexsym}
\usepackage{graphicx}
\usepackage{epstopdf}
\usepackage{indentfirst}

\newtheorem{theorem}{Theorem}
\newtheorem{definition}[theorem]{Definition}
\newtheorem{lemma}[theorem]{Lemma}

\newtheorem{prob}[theorem]{Problem}

\newtheorem{corollary}[theorem]{Corollary}
\newtheorem{conj}[theorem]{Conjecture}
\newtheorem{cl}[theorem]{Claim}

\def\cA{{\mathcal A}}
\def\cC{{\mathcal C}}

\def\cG{{\mathcal G}}
\def\qed
 {\ifhmode\unskip\nobreak\hfill$\Box$\medskip\fi
 \ifmmode\eqno{\Box}\fi}

\def\cL{{\mathcal L}}
\def\cM{{\mathcal M}}

\def\Gn{{\mathcal G}(n)}
\def\Gnn{{\mathcal G}(n, n)}
\def\Gnp{{\mathcal G}(n, p)}
\def\Gnnp{{\mathcal G}(n, n, p)}
\def\Gnkmin{{\mathcal G}(n, k, {\rm min})}
\def\Gnnkmin{{\mathcal G}(n, n, k, {\rm min})}
\def\Gnkmax{{\mathcal G}(n, k, {\rm max})}
\def\Gnnkmax{{\mathcal G}(n, n, k, {\rm max})}
\def\Gnkavg{{\mathcal G}(n, k, {\rm avg})}
\def\Gnnkavg{{\mathcal G}(n, n, k, {\rm avg})}
\def\GnkKn{{\mathcal G}(n, k, {\rm Kneser})}

\def\rmin{{f_{\min}}}
\def\rmax{{f_{\max}}}
\DeclareMathOperator{\avg}{avg}
\DeclareMathOperator{\diam}{diam}
\def\ravg{{f_{\avg}}}
\def\fKn{{f_{\rm Kneser}}}
    \def\Kne{{\rm Kneser}}
\def\fPra{{f_{\rm Pra}}}

\begin{document}

\normalsize

\title{{\bf \huge Kneser ranks of random graphs and minimum difference representations}}

\pagestyle{myheadings} \markright{{\small{\sc Z.~F\"uredi and I.~Kantor:    Kneser ranks and min-difference representations}}}  

\author{\bf Zolt\'an F\"uredi${}^{1}$%
\thanks{Research was supported in part by grant (no. K116769)  
 from the National Research, Development and Innovation Office – NKFIH,
by the Simons Foundation Collaboration Grant \#317487,
and by the European Research Council Advanced Investigators Grant 267195.}\,\,
    and  Ida Kantor${}^2$%
		\thanks{Supported by project 16-01602Y of the Czech Science Foundation (GACR).}}

\date{
{$^1$} Alfr\'ed R\'enyi Institute of Mathematics, 
Budapest, Hungary \\
(e-mail: \texttt{z-furedi@illinois.edu})\\
{$^2$} Charles University, Prague\\
(e-mail: \texttt{idasve@gmail.com})
}

\footnotetext{
 \emph{2010 Mathematics Subject Classification}:
  05C62, 05C80.\hfill    {\tt  [{\jobname}.tex]}\newline
 Submitted to  \emph{???}   \hfill
 Printed on \today }

\maketitle

\renewcommand{\thefootnote}{\empty}


\begin{abstract}
Every graph $G=(V,E)$ is an induced subgraph of some Kneser graph of rank $k$, i.e., there is an assignment of (distinct)  $k$-sets $v \mapsto A_v$ to the vertices $v\in V$ such that
 $A_u$ and $A_v$ are disjoint if and only if $uv\in E$.
The smallest such $k$ is called the {\em Kneser rank} of $G$ and denoted by $\fKn(G)$.
As an application of a result of Frieze and Reed concerning the clique cover number of random graphs we show that
 for constant $0< p< 1$ there exist constants $c_i=c_i(p)>0$, $i=1,2$ such that with high probability
\[   c_1 n/(\log n)< \fKn(G) < c_2 n/(\log n).
\]
We apply this for other graph representations defined by Boros, Gurvich and Meshulam.

A {\em $k$-min-difference representation} of a graph $G$ is an assignment of a set $A_i$ to each vertex $i\in V(G)$ such that
\[ ij\in E(G) \,\, \Leftrightarrow \, \, \min \{|A_i\setminus A_j|,|A_j\setminus A_i| \}\geq k.
\]  
The smallest $k$ such that there exists a $k$-min-difference representation of $G$ is denoted by $f_{\min}(G)$. 
Balogh and Prince proved in 2009 that for every $k$ there is a graph $G$ with $f_{\min}(G)\geq k$. 
We prove that there are constants $c''_1, c''_2>0$ such that  $c''_1 n/(\log n)< f_{\min}(G) < c''_2n/(\log n)$ holds for almost all  bipartite graphs $G$ on $n+n$ vertices. 

\end{abstract}


\section{Kneser representations}

A representation of a graph $G$ is an assignment of mathematical objects of a given kind (intervals, disks in the plane, finite sets, vectors, etc.) to the vertices of $G$ in such a way that two vertices are adjacent if and only if the corresponding sets satisfy a certain condition (intervals intersect, vectors have different entries in each coordinate, etc.). Representations of various kinds have been studied extensively,  see, e.g., \cite{EatonRodl},  \cite{fu90}, the monograph~\cite{McKeeMcMorris}, or from information theory point of view~\cite{KornerMonti01}.
The representations considered in this paper are assignments $v \mapsto A_v$ to the vertices $v\in V$ of a graph $G=(V,E)$ such that the $A_v$'s are (finite) sets satisfying certain relations.

The {\em Kneser graph} ${\rm Kn}(s,k)$ (for positive integers $s\geq 2k$) is a graph whose vertices are all the $k$-subsets of the set $[s]:=\{1,2, \dots, s\}$, and whose edges connect two sets if they are disjoint.
An assignment $(A_1, \dots , A_n)$ for a graph $G=(V,E)$ (where $V=[n]$) is called a {\em Kneser representation of rank} $k$ if
 each $A_i$ has size $k$, the sets are distinct, and  $A_u$ and $A_v$ are disjoint if and only if $uv\in E$.

Every graph on $n$ vertices with minimum degree $\delta< n-2$ has a Kneser representation of rank $(n-1-\delta)$.
To see that, define the {\em co-star} representation $(A_1', \dots, A_n')$ of $G$.
For every $i\in V(G)$, let $A'_i$ be the set of the edges adjacent to $i$ in the complement of $G$ (this is the graph $\overline{G}$ with $V(\overline{G})=V(G)$ and $E(\overline{G})=\binom{V(G)}{2}\setminus E(G)$).  
We have $A'_i\cap A'_j =1$ if $ij\not \in E(G)$, otherwise $A'_i\cap A'_j =0$, and the maximum size of $A'_i$ is $n-1-\delta(G)$. 
To turn the co-star representation into a Kneser representation add pairwise disjoint sets of labels to the sets $A'_1,\dots,A'_n$ to increase their cardinality to exactly $n-1-\delta(G)$. 
The resulting sets $A_1,\dots,A_n$ are all distinct, they have the same intersection properties as $A'_1,\dots,A'_n$, and form a Kneser representation of $G$ of rank $n-1-\delta(G)$.

Let $\Gn$ denote the set of $2^{\binom{n}{2}}$ (labelled) graphs on $[n]$ and let $\GnkKn$ denote the family of graphs on $[n]$ having a Kneser representation of rank $k$.
$G\in \GnkKn$ is equivalent to the fact that $G$ is an {\em induced} subgraph of some Kneser graph ${\rm Kn}(s,k)$.
We have
\[   {\mathcal G}(n, 1, {\Kne})\subseteq  {\mathcal G}(n, 2, {\Kne}) \subseteq \dots \subseteq {\mathcal G}(n, n-1, {\Kne})= \Gn.
\]
Let $\fKn(G)$ denote the smallest $k$ such that $G$ has a Kneser representation of rank $k$. We have seen that $ \fKn(G)\leq n-\delta$.
We show that there are better bounds for almost all graphs.

\begin{theorem}\label{th:1}
There exist constants $c_2 > c_1>0$ such that for $G\in \Gn$ with high probability
\[   c_{1} \frac{n}{\log n}< \fKn (G) < c_{2}\frac{n}{\log n}.
\]
  \end{theorem}

We will prove a stronger version as Corollary~\ref{co:1}.

\section{Minimum difference representations}

In {\em difference representations}, generally speaking, vertices are adjacent if the representing sets are sufficiently different. As an example consider Kneser graphs, where the vertices are adjacent if and only if the representing sets are disjoint.
There are other type of representations where one joins sets close to each other, e.g.,   $t$-{\em intersection representations} were investigated by {M.~Chung and West}~\cite{MChungWest} for dense graphs and {Eaton and R{\"o}dl}~\cite{EatonRodl} for sparse graphs. But these are usually lead to different type of problems, one cannot simply consider the complement of the graph.


This paper is mostly focused on $k$-min-difference representations (and its relatives), defined by Boros, Gurvich and Meshulam in~\cite{BGM04} as follows.

\begin{definition}
Let 
 $G$ be a graph on the vertices $[n]=\{ 1,\dots,n\}$.
A {\em $k$-min-difference representation} $(A_1,\dots,A_n)$
of $G$ is an assignment of a set $A_i$ to each vertex $i\in V(G)$ so that
\[ ij\in E(G) \,\, \Leftrightarrow \,\, \min \{|A_i\setminus A_j|,|A_j\setminus A_i| \}\geq k.
\]
Let $\cG(n,k,\min)$ be the set of graphs with $V(G)=[n]$ that have a $k$-min-difference representation.
The smallest $k$ such that 
$G\in \cG(n,k,\min)$
is denoted by $\rmin(G)$.
\end{definition}
The co-star representation (which was investigated by {Erd{\H{o}}s,  Goodman, and P{\'o}sa}~\cite{EGP66} in their classical work on clique decompositions)
  shows that $\rmin(G)$ exists and it is at most  $n-1-\delta(G)$.

Boros, Collado, Gurvits, and Kelmans~\cite{BCGK00} showed that many $n$-vertex graphs, including all trees, cycles, and line graphs, the complements of the above, and $P_4$-free graphs, belong to $\cG(n,2,\min)$.
They did not find any graph with $\rmin(G)\geq 3$.
Boros, Gurvitch and Meshulam~\cite{BGM04} asked whether the value of $\rmin$ over all graphs is bounded by a constant. This question was answered in the negative by
Balogh and Prince~\cite{BP09}, who proved that for every $k$  there is an $n_0$ such that whenever $n>n_0$, then for a graph $G$ on $n$ vertices we have $\rmin(G)\geq k$ with high probability.
Their proof used a highly non-trivial Ramsey-type result due to {Balogh and Bollob{\'a}s}~\cite{baloboll:unav05}, so their bound on $n_0$ is a tower function of $k$.

Our main result is a significant improvement of the Balogh-Prince result.
Let $\Gnn$ denote the family of $2^{n^2}$ bipartite graphs $G$ with partite sets $V_1$ and $V_2$, $|V_1|=|V_2|=n$.
\begin{theorem}\label{main}
There is a constant $c>0$ such that
for almost all bipartite graphs $G\in \Gnn$ one has $\rmin(G)\geq cn/{(\log n)}$.
\end{theorem}

Let $H$ be a graph on $\log n$ vertices with $\rmin(H)\geq c\log n/{(\log \log n)}$.
One of the basic facts about random graphs is that almost all graphs on $n$ vertices contain $H$ as an induced subgraph. The following theorem is an easy consequence of this fact together with Theorem~\ref{main}.

\begin{corollary}\label{co:probversion}
There is a constant $c>0$ 
such that   
almost all graphs $G$ on $n$ vertices satisfy
\[\rmin(G) \geq \frac{c\log n}{\log \log n}. \]
\end{corollary}


\section{On the number of graphs with $k$--min-dif representations}

\subsection{The structure of min-dif representations of bipartite graphs}

Analogously to previous notation,
   $\Gnkmin$ (and  $\Gnnkmin$) denotes the family of (bipartite) graphs $G$ with $n$ labelled vertices $V$ (partite sets $V_1$ and $V_2$, $|V_1|=|V_2|=n$, respectively)
with $\rmin(G)\leq k$.
Our aim in this Section is to show that there exists a constant $c> 0$ such that $|\Gnnkmin|= o(2^{n^2})$ if $k < c n /(\log n)$.
This implies  that for almost all bipartite graphs on $n+n$ vertices $\rmin(G) \geq c n/ (\log n)$.

A $k$-min-difference representation $( A_i: i\in V)$ of $G$ is {\em reduced} if deleting any element $x$ from all sets that contain it yields a representation  of a graph different from $G$.
Note that
\[     |A_i\setminus A_j|-1\leq   |(A_i\setminus x)\setminus (A_j\setminus x)| \leq |A_i\setminus A_j|
\]
so the graph $G'$ corresponding to the $k$-representation $( A_i\setminus x: i\in V)$ has no more edges than $G$,  $E(G')\subseteq E(G)$.
There is a natural partition of the elements of $\bigcup A_i$: for every $\emptyset \neq I\subseteq [n]$, we have the subset ($\bigcap_{i\in I}A_i) \cap (\bigcap_{j\not\in I}\overline{A_j})$ where $\overline{A_j}$ is the complement of the set $A_j$. We call these subsets {\em atoms}. If a $k$-min-difference representation is reduced, then no atom has more than $k$ elements. It follows that the ground set $\bigcup A_i$ of a reduced representation of an $n$-vertex graph has no more than $k 2^n$ elements. Lemma~\ref{l:noelts} improves on this observation.

\begin{lemma}\label{l:noelts}
Let $G$ be a graph with $n$ vertices and $(A_1,\dots, A_n)$ a reduced $k$-min-difference representation of $G$. Then
\[\left\lvert\bigcup A_i \right\rvert \leq 2e(G)k \leq k n^2.
\]
\end{lemma}

\begin{proof}
Define the sets $A_{i,j}:= A_i\setminus A_j$ in the cases $ij \in E(G),$ and $|A_i\setminus A_j|=k$.
Let $S:= \bigcup A_{i,j}$. The number of elements in $S$ is bounded above by the quantity $|E(G)|\cdot 2k$.
We claim that $S=\bigcup A_i$. Otherwise, if there is an element  $x\in \left(\bigcup A_i \right)\setminus S$, then
  the representation can be reduced, $( A_i\setminus x: i\in V)$ defines the same graph as $( A_i: i\in V)$.
\end{proof}

The upper bound in Lemma~\ref{l:noelts} can be significantly improved for bipartite graphs.

\begin{figure}
\vspace{-11.6cm}
\begin{center}
\includegraphics[scale=.6]{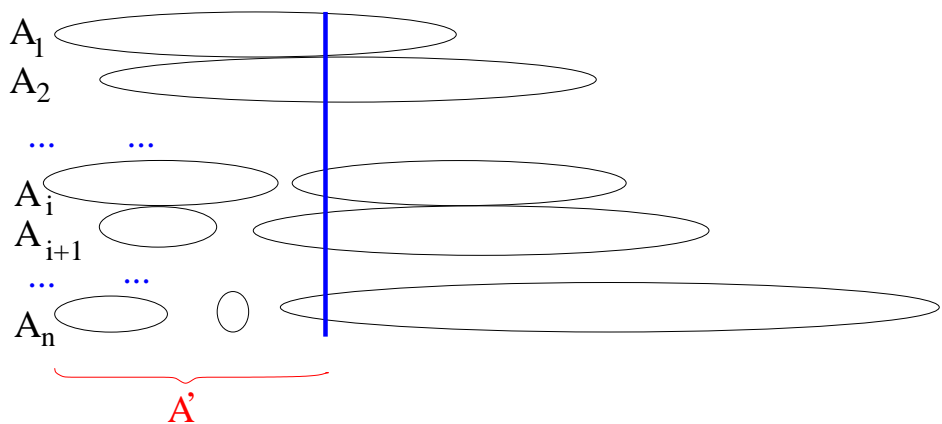}
\end{center}
\caption{$|A_1|\leq \dots \leq |A_n|$ in a min-dif representation when $\{ 1,2, \dots, n\}$ is independent}
\label{figAn}
\end{figure}

\begin{lemma}\label{l:noeltsBIP}
Let $G\in \Gnn$ be a bipartite graph with $n+n$ labeled vertices, $G\in \Gnnkmin$.  Let $(A_1,\dots,A_n)$ and $(B_1,\dots,B_n)$ be the sets representing the two parts.
If $(A_1,\dots, A_n, B_1, \dots , B_n)$ is a reduced $k$-min-difference representation of $G$, then
\[\left\lvert\left(\bigcup A_i\right) \cup \left(\bigcup B_i\right)   \right\rvert \leq  4kn.
\]
\end{lemma}

\begin{proof}
Suppose that $|A_1|\leq \dots \leq |A_n|$ and $|B_1|\leq \dots \leq |B_n|$. Let $A:= \bigcup A_i$ and $B:=\bigcup B_i$, $S:=A\cup B$.
Define
\begin{equation}\label{eq:A'}
A':= \bigcup_{i=1}^{n-1}(A_i\setminus A_{i+1}).
\end{equation}
For each $i$, the inequality $|A_i\setminus A_{i+1}|\leq |A_{i+1}\setminus A_i|$ follows from the assumption that $|A_i|\leq |A_{i+1}|$.
The vertices in each part of $G$ form an independent set, so for each $i$, we have $|A_i\setminus A_{i+1}|\leq k-1$.
Hence $|A'|\leq (n-1)k$.

If $x\in A_{\alpha}\setminus A_{\beta}$ for some $\alpha < \beta$, then there is an index $i$ such that $x\in A_i\setminus A_{i+1}$ and therefore $x\in A'$.
 In other words, if $x\in A_\alpha \setminus A'$ and $\alpha < \beta$ then $x\in A_\beta$.
Therefore the  sets $A_i\setminus A'$ form a chain (see Figure~\ref{figAn}),
\[  A_1\setminus A' \subseteq A_2\setminus A' \subseteq \dots \subseteq A_n\setminus A'.
\]
Treat the other part of $G$ analogously: define $B'$ and note the same bound on its size, and note that the sets $B_i\setminus B'$ form a chain.

Let us define $D=S\setminus (A'\cup B')$.
We will prove that there are at most $2(n+1)$ sets of the form
  $A_m\setminus B_{\ell}$ and $B_{p}\setminus A_q$, each of cardinality $k$, covering $D$. Therefore $D$ contains at most $2(n+1)k$ elements.
For each $1\leq i\leq n$, let us define $\widetilde{A_i}=A_i\cap D$ and $\widetilde{B_i}=B_i\cap D$.
Let $\widetilde{A_0}=\widetilde{B_0}=\emptyset$ and $\widetilde{A_{n+1}}=\widetilde{B_{n+1}}=D$.
The sets $\widetilde{A_0}, \widetilde{A_1},\dots,\widetilde{A_n}, \widetilde{A_{n+1}}$ form a chain, same for $\widetilde{B_0},\widetilde{B_1}\dots,\widetilde{B_n},\widetilde{B_{n+1}}$. The elements of $D$ belong to $(n+1)^2$ atoms (as defined in the beginning of this section), many of them possibly empty, corresponding to the squares in Figure~\ref{fig1}.

\begin{figure}
\vspace{0.7cm}
\begin{center}
\includegraphics[scale=.9]{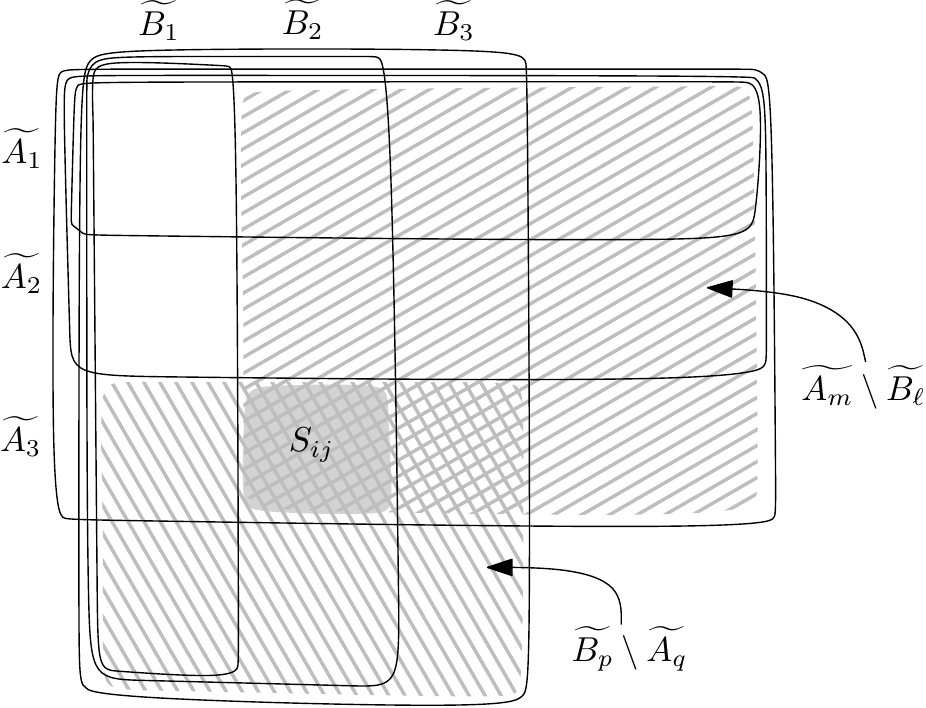}
\end{center}
\caption{The elements of $S\setminus (A'\cup B')$ split into $(n+1)^2$ atoms. }
\label{fig1}
\end{figure}

For each $i,j$, $1\leq i,j \leq n+1$, $(i,j)\neq (n+1,n+1)$, the atom $S_{i,j}$ is defined as $(\widetilde{A_i}\setminus \widetilde{A_{i-1}})\cap (\widetilde{B_j}\setminus \widetilde{B_{j-1}})$. 
Since the representation is reduced, no elements from the atom $S_{i,j}$ can be left out, so e.g., $S_{1,1}=\emptyset$. 
It follows that either there are some $m$ and $\ell$ such that $|A_m\setminus B_{\ell}|=k$ and the atom $S_{i,j}$ belongs in $A_m\setminus B_{\ell}$
(here $n\geq m\geq i\geq 1$ and $j> \ell \geq 1$),
or there are some $p,q$ such that $|B_{p}\setminus A_q|=k$ and the atom $S_{i,j}$ is in $B_{p}\setminus A_q$. Since $|\widetilde{A_m}\setminus \widetilde{B_{\ell}}|\subseteq A_m\setminus B_{\ell}$, we have $|\widetilde{A_m}\setminus \widetilde{B_{\ell}}|\leq k$ in the first case.  
Likewise in the second case, $|\widetilde{B_{p}}\setminus \widetilde{A_q}|\leq k$. In Figure~\ref{fig1}, the first option corresponds to a rectangle containing the $S_{i,j}$ cell and the upper-right corner, with all the squares in this rectangle together containing only at most $k$ elements. 
The second option corresponds to a similar rectangle with only at most $k$ elements in it, containing the $S_{i,j}$ square and the lower-left corner.

Call a subrectangle $\widetilde{A_m}\setminus \widetilde{B_{\ell}}$ {\em critical} if $|A_m\setminus B_{\ell}|=k$,
  and similarly $\widetilde{B_{p}}\setminus \widetilde{A_q}$ is critical if $|B_{p}\setminus A_q|=k$.
Our argument above can be reformulated that every (nonempty) cell $S_{i,j}$ is covered by a critical rectangle.
This implies that in each row one can find at most two critical rectangles that cover all non-empty atoms in it.
This yields the desired upper bound $|D|\leq 2(n+1)k$.

Finally, altogether $|S|\leq |A'|+|B'|+|D| \leq 4kn$.
\end{proof}

\subsection{Counting reduced matrices}

Let $S$ be a set of size $|S|=4kn$.
In this subsection we give an upper bound for the number of sequences $(A_1, \dots, A_n)$ of subsets of $S$
satisfying the following two properties
\\${}$\quad (P1) \quad    $|A_1|\leq \dots \leq |A_n|$,
\\${}$\quad (P2) \quad    $|A_i\setminus A_{i+1}|\leq k-1$  \enskip(for all $1\leq i \leq n-1$).

Let $\cM$ be the $0$-$1$ matrix that has the characteristic vectors of the sets $A_1,\dots, A_n$ as its rows (in this order).
The positions in $\cM$ where an entry 1 is directly above an entry 0 will be called {\em one-zero configurations}, while the positions where a 0 is directly above a 1 will be called {\em zero-one configurations}.
A column in a 0-1 matrix is uniquely determined by the locations of the one-zero configurations and the zero-one configurations unless it is a full 0 or full 1 column.
We count the number of possible matrices $\cM$ by filling up the $n\times (4kn)$ entries in three steps.

Each one-zero configuration corresponds to an $i< n$ and to an element $x\in A_i\setminus A_{i+1}$.
A set $A_{i+1}\setminus A_i$ can be selected in at most
\[
\binom{4nk}{0}+ \dots + \binom{4nk}{k-1}< (4en)^k
\]
ways ($n> k\geq 1$).
Do this for each $i<n$, altogether we have less than $(4en)^{kn}$ ways to write in the  one-zero configurations into $\cM$.

Select in each column the top $1$.
If there is no such element in a column we indicate that it is blank, a full zero column.
There are at most $n+1$ outcomes for each column, altogether there are at most $(n+1)^{4kn}$ possibilities.
Fill up with $0$'s each column above its top $1$.
Define  $A'\subset S$ as in \eqref{eq:A'}, $A':= \bigcup_{i=1}^{n-1}(A_i\setminus A_{i+1})$.
We have $|A'|\leq kn$.
 The columns of $\cM$ that correspond to the elements of $S\setminus A'$ have a (possibly empty) string of zeros followed by a string of ones.
We almost filled up $\cM$ and we can finish this process by selecting the remaining zero-one configurations.

There may be several zero-one configurations in a single column.
Each of them has a unique (closest, or smallest indexed) $1$ above them.
That element $1$ is already written in into our still partially filled $\cM$, because that element 1 (even if it is the top $1$ element) belongs to a unique one-zero configuration.
This correspondence is an injection.
So there are at most $\sum_i |A_i\setminus A_{i+1}| \leq kn$ zero-one configurations in the columns corresponding to $A'$ which are not yet identified.
There are at most $n^{kn}$ ways to select them.

Since (for $n> k\geq 1$)
\begin{equation*}\label{eq:M}
(4en)^{kn} \times (n+1)^{4kn}\times n^{kn} < n^{6kn+ O(kn/\log n)}= e^{6kn\log n + O(kn)},
\end{equation*}
  we obtain the following

\begin{cl}\label{cl:M}
Altogether, there are $e^{O(kn \log n)}$ ways  to fill $\cM$ with entries in $\{0,1\}$ according to the rules {\rm (P1)} and {\rm (P2)}.
   \end{cl}

\subsection{Proofs of the lower bounds}

\begin{proof}[Proof of the lower bound in Theorem~\ref{main}.]
Let $G=G(V_1, V_2)$ be a bipartite graph with both parts of size $n$ and suppose that $G$ belongs to $\cG(n,n,k,\min)$.
By Lemma~\ref{l:noeltsBIP} we may suppose that $G$ has a reduced $k$-min-difference representation
 $(A_1,\dots, A_n, B_1, \dots , B_n)$ such that each representing set is a subset of $S$, where $|S|=4kn$.
There are permutations $\pi$ and $\rho\in S_n$ which rearrange the sets according their sizes
$|A_{\pi(1)}|\leq \dots \leq |A_{\pi(n)}|$ and $|B_{\rho(1)}|\leq \dots \leq |B_{\rho(n)}|$.
Consider the $V_i\times S$ matrices, $\cM_i$, $i=1,2$ whose $i$'th row is the $0$-$1$ characteristic vector of $A_{\pi(i)}$ and $B_{\rho(i)}$, respectively.
The permutations $\pi$, $\rho$ and the matrices $\cM_1$, $\cM_2$ completely describe $G$.
The matrices $\cM_i$ satisfies properties (P1) and (P2), so Claim~\ref{cl:M}
  yields the following upper bound  for the number of such fourtuples $(\pi, \rho, \cM_1, \cM_2)$
\begin{equation}\label{eq:Gnnk}   |\cG(n,n,k,\min)| \leq \# {\rm \, of  \,} (\pi, \rho, \cM_1, \cM_2)'{\rm s} \leq (n!)^2 n^{(12+o(1))kn}= e^{O(kn \log n)}.
\end{equation}
Here the right hand side is $o(2^{n^2})$ if $k \leq  0.057 n / (\log n)$ implying that
  $\rmin(G)>0.057 n / (\log n)$ for almost all the $2^{n^2}$ bipartite graphs.
\end{proof}

\begin{proof}[Proof of the lower bound for the random bipartite graph.]${}$\\
Recall that in a random graph $G\in \Gnp$, each of the $\binom{n}{2}$ edges occurs independently with probability $p$.
Similarly, $\Gnnp$ denotes the class of graphs $\Gnn$ with the probability of a given graph $G\in \Gnn$  is
 \[   p^{e(G)}(1-p)^{n^2 - e(G)}.   
  \]
Here the right hand side is at most  $(\max\{ p, 1-p\})^{n^2}$.
This implies that for any class of graphs $\cA\subset \Gnn$  the probability $\Pr (G\in \cA) $
is at most $|\cA|$ times this upper bound. 
If the class of graphs $\cA$ is too small, namely 
\[ 
  |\cA| = o\left((\min\{  \frac{1}{p}, \frac{1}{1-p} \})^{n^2}\right), 
 \] 
 then for $\Gnnp$ one has 
\begin{equation}\label{eq:pr} \Pr (G\in \cA)\to 0. 
\end{equation} 
Taking $\cA:=\Gnnkmin$ with a sufficiently small $k$, we obtain

\begin{corollary}\label{co:6}
For constant $0< p< 1$ there exists a constant $c=c_{1, \rm min}(p)>0$ such that the following holds for $G\in \Gnnp$  with high probability
    as $n\to \infty$
\[  c \frac{n}{\log n} < \rmin (G).
\]
  \end{corollary}
\end{proof}

\section{Maximum and average difference representations}

Boros, Gurvich and Meshulam~\cite{BGM04} also defined $k$-{\em max-difference} representations and $k$-{\em average-difference} representations of a graph $G$ in a natural way, that is, the vertices $i$ and $j$ are adjacent if and only if for the corresponding sets $A_i$, $A_j$ we have $\max\{|A_i\setminus A_j|,|A_j\setminus A_i|\}\geq k$ and $(|A_i\setminus A_j|+|A_j\setminus A_i|)/2\geq k$, respectively.  
Analogously to $\rmin$ we can define $f_{\max}(G)$ and $f_{\avg}(G)$. 
Since for every graph $G$ a Kneser representation is a min-dif, average-difference, and max-difference representation as well we get
\begin{equation}\label{eq:19-}   \rmin(G),\, \ravg(G), \, \rmax(G) \leq \fKn(G) \leq n-1.
\end{equation}
Let    $\Gnkmax$, $\Gnnkmax$, ($\Gnkavg$,  $\Gnnkavg$) denote the family of graphs $G\in \Gn$ and in $\Gnn$ with
 $n$ labeled vertices $V$  or with partite sets $V_1$ and $V_2$, $|V_1|=|V_2|=n$, respectively, such that $\rmax(G)\leq k$ ($\ravg(G)\leq k$, respectively).

It was proved in~\cite{BGM04} that $f_{\max}$ and $f_{\min}$ are not bounded by a constant,  for a matching of size $t$ one has $f_{\max}(tK_2)= \Theta(\log t)$ and $f_{\avg}(tK_2)= \Theta(\log t)$. (It turns out that $f_{\min}(tK_2)=1$.) The proof of Theorem~\ref{main} can be easily adapted for these parameters  for $\Gnnp$ as well. 
Even more, we can handle the general case $G\in \Gnp$, too.

\begin{corollary}\label{co:7}
For constant $0< p< 1$ there exists a constant $c=c(p)>0$ such that the following holds for $G\in \Gnnp$  with high probability as $n\to \infty$
 \begin{equation}\label{eq:Knn_lower}
  c \frac{n}{\log n} <  \ravg (G), \,  \rmax (G) .
  \end{equation}
Similarly for $G\in \Gnp$ with high probability we have
 \begin{equation}\label{eq:Kn_lower}
    c \frac{n}{\log n} <    \ravg (G), \,  \rmax (G).
  \end{equation}
  \end{corollary}

These lower bounds together with the upper bounds from Corollary~\ref{co:3} below imply that for almost all
  $n$ vertex graphs, and for almost all bipartite graphs on $n+n$ vertices, $\ravg(G)$ and $\rmax(G)$ are $\Theta(n/(\log n))$.

\begin{proof}[Sketch of the proof.]
If $G\in \Gnnkmax$ (and if $G\in \Gnnkavg$) and  $(A_1, \dots, A_n, B_1,$ $ \dots, B_n)$ is a $k$-max-difference
($k$-average-difference) representation
then
\begin{equation}\label{eq:9}
  |A_i \bigtriangleup A_j| \leq 2k-2  \enskip (\leq 2k-1)
  \end{equation}
holds for each pair $i,j$.
If the representation is reduced, then we obtain (without the tricky proof of Lemma~\ref{l:noeltsBIP}) that $|\bigcup_i A_i|< 2kn$, and the same holds for
$|\bigcup_i B_i|$, too.
The conditions of Claim~\ref{cl:M} are satisfied implying
\[   |\cG(n,n,k,\max)|, |\Gnnkavg| = e^{O(kn \log n)}.
\]
We complete the proof of~\eqref{eq:Knn_lower} applying~\eqref{eq:pr} as it was done at the end of the previous Section.

Consider a graph  $G\in \Gnkmax$ and let  $(A_1, \dots, A_n)$ be a reduced $k$-max-difference  representation.
(The case of $k$-average-difference representation can be handled in the same way, and the details are left to the reader).
The only additional observation we need is that since~\eqref{eq:9} holds for
  each non-edge $\{i,j\}$, we have $ |A_i \bigtriangleup A_j| \leq 4k-4$  for {\em all} pairs of vertices whenever $\diam (\overline{G})\leq 2$.
Thus for every reduced representation (in case of  $\diam (\overline{G})\leq 2$) one has $|\bigcup_i A_i|\leq (4k-4)n$.
Also, $|A_i\setminus A_j| \leq 2k-2$ for $|A_i|\leq |A_j|$.
Then the conditions of Claim~\ref{cl:M} are fulfilled (with $2k-2$ instead of $k$) implying the following version of~\eqref{eq:Gnnk}
\[  |\cG(n,n,k,\min)\setminus \cG_2(n)| = e^{O(kn \log n)},
\]
where $\cG_2(n)$ denotes the class of graphs with $G\in \Gn, \diam (\overline{G})> 2$.

We complete the proof of~\eqref{eq:Kn_lower} by applying~\eqref{eq:pr} and the fact that
\[ \diam(G)\geq 2
\]
 holds with high probability for $G\in \Gnp$.
\end{proof}


\section{Clique covers of the edge sets of graphs}

We need the following version of Chernoff's inequality (see, e.g.,~\cite{AlonSpencer4}).
Let $Y_1, \dots, Y_n$ be mutually independent random variables with $E[Y_i]\leq 0$ and all $|Y_i|\leq 1$. Let $a \geq 0$. Then
\begin{equation}\label{Chernoff}
   \Pr [Y_1+ \dots + Y_n > a] < e^{-a^2/(2n)}.
 \end{equation}

A finite {\em linear space} is a pair $(P, \cL)$ consisting of a set $P$ of elements (called {\em points}) and a set $\cL$ of subsets of $P$ (called {\em lines}) satisfying the following two properties.\\
${}$\hspace{0.8cm}$\rm (L1)$\quad Any two distinct points $x, y \in P$ belong to exactly one line $L=L(x,y)\in \cL$.
\\
${}$\hspace{0.8cm}$\rm (L2)$\quad Any line has at least two points. \\
In other words, the edge set of the complete graph $K(P)$ has a clique decomposition into the complete graphs $K(L)$, $L\in \cL$.

\begin{lemma}\label{l:linear}
For every positive integer $n$ there exists a linear space $\cL=\cL_n$ with lines $L_1, \dots, L_m$ such that
$m= n+ o(n)$, every edge has size $(1+o(1))\sqrt{n}$, and every point belongs to $(1+o(1))\sqrt{n}$ lines.
\end{lemma}

\begin{proof} (Folklore). If $n=q^2$ where $q>1$ is a power of a prime then  we can take the $q^2+q$ lines of an affine geometry $AG(2,q)$.
Each line has exactly $q=\sqrt{n}$ points and each point belongs to $q+1$ lines.
In general, one can consider the smallest power of prime $q$ with $n\leq q^2$ (we have $q=(1+o(1))\sqrt{n}$) and
  take a random $n$-set $P\subset {\mathbb F}_q^2$ and the lines defined as $P\cap L$, $L\in \cL(AG(2,q))$.
\end{proof}

\subsection{Thickness of clique covers}
The {\em clique cover number} $\theta_1(G)$ of a graph $G$ is the minimum number of cliques required to cover the edges of graph $G$.
Frieze and Reed~\cite{FriezeReed} proved that for $p$ constant, $0< p<1$, there exist constants $c'_i=c'_i(p)>0$, $i=1,2$ such that for $G\in \Gnp$ with high probability
\[   c'_1\frac{n^2}{(\log n)^2} < \theta_1(G) <     c'_2\frac{n^2}{(\log n)^2}.
\]
They note that `a simple use of a martingale tail inequality shows that $\theta_1$ is close to its mean with very high probability'. 
We only need the following consequence concerning the expected value. 
 \begin{equation}\label{eq:FR2} 
   E(\theta_1(G)) <     c'_3\frac{n^2}{(\log n)^2}. 
  \end{equation} 

The {\em thickness} $\theta_0$ of a clique cover $\cC:= \{ C_1, \dots, C_m\}$ of $G$ is the maximum degree of the hypergraph $\cC$, i.e.,
 $\theta_0(\cC):= \max_{v\in V(G)} \deg_\cC(v)$.
The minimum thickness among the clique covers of $G$ is denoted by $\theta_0(G)$.

A clique cover $\cC$ corresponds to a set representation  $v\mapsto A_v$ in a natural way $A_v:= \{ C_i: v\in C_i\}$ with the property that
 $A_u$ and $A_v$ are disjoint if and only if $\{u,v\}$ is a non-edge of $G$.
The size of the largest $A_v$ is the thickness of $\cC$.
For $k> \theta_0(\cC)$ one can add $k-|A_v|$ distinct extra elements to $A_v$ (for each $v\in V(G)$), thus obtaining a
   Kneser representation of rank $k$ of the complement of $G$, $\overline{G}$.
This relation can be reversed, yielding
\begin{equation}\label{eq:thickness2}
      \theta_0(G) \leq     \fKn(\overline{G})  \leq \theta_0(G) +1.
  \end{equation}

\begin{theorem}\label{th:thickness}
For constant $0< p< 1$ there exist constants $c_i=c_i(p)>0$, $i=1,2$ such that for $G\in \Gnp$ with high probability
\[   c_1 \frac{n}{\log n}< \theta_0(G) < c_2\frac{n}{\log n}.
\]
\end{theorem}
\begin{proof}
The lower bound is easy.
The maximum degree $\Delta(G)$ of $G \in \Gnp$ with high probability satisfies $\Delta\approx np$.
As usual we write $a_n \approx b_n$ as $n$ tends to infinity.
Also for the unproved but well-known statements concerning the random graphs see the monograph~\cite{JLR_book}.
The size of the largest clique $\omega=\omega(G)$ with high probability satisfies $\omega \approx 2 \log n/ (\log (1/p))$.
Since $\theta_0\geq \Delta/(\omega-1)$ we may choose $c_1\approx p (\log (1/p))/2$.

The upper bound probably can be proved by analyzing and redoing the clever proof of Frieze and Reed concerning $\theta_1(G_{n,p})$.
Probably their randomized algorithm yields the upper bound for the thickness, too.
Although there are steps in their proof where they remove from $G$ (as cliques of size 2) an edge set of size $O(n^{31/16})$ and
  one needs to show that these edges are well-distributed.
However, one can easily deduce the upper bound for $\theta_0(G_{n,p})$ directly only from Equation~\eqref{eq:FR2}.

Given $n$, fix a  linear hypergraph $\cL=\cL_n$ with point set $[n]$ and hyperedges $L_1, \dots, L_m$ provided by
Lemma~\ref{l:linear}.
We have  $m= n+ o(n)$, $\ell_i:=|L_i|=(1+o(1))\sqrt{n}$, and every point $v$ belongs to $b_v=(1+o(1))\sqrt{n}$ lines.
Build the random graph $G\in\Gnp$ in $m$ steps by taking a $G_i\in {\mathcal G}(L_i, p)$.
Let $\cC_i$ be a clique cover of $G_i$ with $\theta_1(G_i)$ members, $\cC= \cup_{1\leq i\leq m} \cC_i$.

Let $X_i(v)$ denote the thickness of $\cC_i$ at the point $v\in L_i$.
We consider $X_i(v)$ as a random variable, whose distribution is depending only on $\ell_i$ and $p$.
For every $L_i$, we have
\[    \sum_{v; v\in L_i} X_i(v)= \sum _{C\in \cC_i} |C| \leq \theta_1(G_i) \omega(G_i).
  \]
Here  with {\em very} high probability $\omega_i=\omega(G_i)$ satisfies $\omega_i \approx 2 \log  \ell_i/ (\log (1/p)) $.
Then~\eqref{eq:FR2} implies that
\begin{equation*} E(\sum_{v\in L_i}  X_i(v))  \leq   c'_3\frac{\ell_i^2}{(\log \ell_i)^2} \times (1+o(1)) \frac{2 \log  \ell_i}{ \log (1/p)}.
\end{equation*}
Since the distributions of $X_i(u)$ and $X_i(v)$ are identical (for $u,v\in L_i$), and there are $\ell_i$ terms on the left hand side, we obtain that
\[
   E(X_i(v))\leq (1+o(1))   \frac{2c'_3}{ \log (1/p)} \times   \frac{\ell_i}{\log \ell_i} < c \frac{\sqrt {n}}{\log n}.
\]
Here we chose $c> 4c_3'/ (\log (1/p))$.

Let $X(v)$ be the thickness of $\cC$ at $v$.
We have $X(v)= \sum _{L_i \ni v} X_i(v)$, where this is a sum of $b:=b_v= (1+o(1))\sqrt{n}$ mutually independent random variables, and each term
is non-negative and is bounded by $\ell= \max_i \ell_i=  (1+o(1))\sqrt{n}$.
Define $b$ independent random variables
\[Y_i:= \frac{1}{\ell} \left(X_i(v) - c\frac{\sqrt{n}}{\log n}\right)
\]
for each $i$ with $L_i\ni v$.
We can apply Chernoff's inequality~\eqref{Chernoff} for any real $a>0$
\[   \Pr\left [\left(\sum _{L_i\ni v} Y_i \right)> a\right] < e^{-a^2/(2b)}.
\]
Substituting $a:= \sqrt{4b \log n}$ the right hand side is $1/n^2$ and we get
\[   \Pr \left[X(v) > c \frac{b \sqrt{n}}{\log n}+  \ell\sqrt{4b\log n}\right] < \frac{1}{n^2}.
\]
Since this is true for all $v\in [n]$, we obtain that (for large enough $n$) for any $c_2> c$
\[   \Pr \left[X(v) < c_2 \frac{n}{\log n}\text{ for all }v\right] > 1- \frac{1}{n},
\]
 completing the proof of the upper bound for $\theta_0(G)$.
\end{proof}

Since the complement of a random graph $G\in \Gnp$ is a random graph from ${\mathcal G}(n, 1-p)$,
Theorem~\ref{th:thickness} and \eqref{eq:thickness2} imply that
\begin{corollary}\label{co:1}
For constant $0< p< 1$, there exist constants $c_{i, \Kne}=c_{i, \Kne}(p)>0$, $i=1,2$ such that for $G\in \Gnp$ with high probability
\[   c_{1, \Kne} \frac{n}{\log n}< \fKn (G) < c_{2, \Kne}\frac{n}{\log n}.
\]
\qed
  \end{corollary}

One can also prove a similar upper bound for the random bipartite graph.
\begin{corollary}\label{co:2}
For constant $0< p< 1$ there exist a constant $c_{3,\Kne}=c_{3, \Kne}(p)>0$ such that for $G\in \Gnnp$ with high probability
\[    \fKn (G) < c_{3, \Kne}\frac{n}{\log n}.
\]
  \end{corollary}

\begin{proof}
Let $G\in \Gnnp$ a random bipartite graph with partite sets $|A|=|B|=n$.
Consider a random graph $G_A\in {\mathcal G}(A, p)$ and $G_B\in {\mathcal G}(B, p)$, their union is
 $H=G\cup G_A\cup G_B$. We can consider $\overline{H}$ as a member of ${\mathcal G}(A\cup B, 1-p)$.
Since $\overline{G}$ can be obtained from $\overline{H}$ by adding two complete graphs $K_A$ and $K_B$, we obtain
\[    \fKn(G)-1 \leq \theta_0(\overline{G})\leq  \theta_0(\overline{H})+1 < 1+ c_2(1-p) \frac{2n}{\log (2n)},
\]
where the last inequality holds with high probability according to Theorem~\ref{th:thickness}.
\end{proof}

Recall~\eqref{eq:19-}, that for every graph $G$,
$ \rmin(G),\, \ravg(G), \, \rmax(G) \leq \fKn(G)$ holds.
These and the above two Corollaries imply the following upper bounds.

\begin{corollary}\label{co:3}
For constant $0< p< 1$ the following holds for $G\in \Gnp$ with high probability as $n\to \infty$
\[     \rmin (G), \,   \ravg (G), \,  \rmax (G) < c_{2, \Kne}(p)\frac{n}{\log n},
\]
and similarly for $G\in \Gnnp$
\[   \rmin (G), \,   \ravg (G), \,  \rmax (G) < c_{3, \Kne}(p)\frac{n}{\log n}.
\]  
  \end{corollary}

\section{Prague dimension}

The {\em Prague dimension} (it is also called {\em product dimension}) $\fPra(G)$ of a graph $G$ is the smallest integer $k$ such that one can find   vertex distinguishing good colorings
 $\varphi_1, \dots, \varphi_k: V(G)\to \mathbb{N}$.
This means that $\varphi_i(u)\neq \varphi_i(v)$ for every edge $uv\in E(G)$ and $1\leq i\leq k$ but for every non-edge $\{ u,v\}$, 
 there exists an $i$ with $\varphi_i(u)=\varphi_i(v)$, 
  moreover the vectors $(\varphi_1(u), \varphi_2(u), \dots, \varphi_k(u))$ and  $(\varphi_1(v), \varphi_2(v), \dots, \varphi_k(v))$ are distict for $u\neq v$.  
Two vertices are adjacent if and only if their labels disagree in every $\varphi_i$.
As Hamburger, Por, and Walsh~\cite{HPW09} observed,  the Kneser rank never exceeds the Prague dimension, so one can extend~\eqref{eq:19-} as follows.
For every graph $G$
\begin{equation}\label{eq:19Pra}   \rmin(G),\, \ravg(G), \, \rmax(G) \leq \fKn(G)\leq \fPra(G).
\end{equation}
The determination of $\fPra(G)$ is usually a notoriously difficult task.
The results of Lov{\'a}sz,  Ne{\v{s}}et{\v{r}}il, and Pultr~\cite{LNPultr} were among the first (non-trivial) applications of the algebraic method.
Hamburger, Por, and Walsh~\cite{HPW09} observed that there are graphs where the difference of $\fPra(G)-\fKn(G)$
 is arbitrarily large, even for Kneser graphs ${\rm Kn}(s,k)$.
Poljak,  Pultr,  and R{\"o}dl~\cite{PoljakPultrRodl} proved  that $\fPra({\rm Kn}(s,k))= \Theta (\log \log s)$ (as $k$ is fixed and $s\to \infty$)
 while $\fKn({\rm Kn}(s,k))=k$ for all $s\geq 2k>0$.
Still we think that for most graphs these parameters have the same order of magnitude.

\begin{conj}\label{con:Pra}
For a constant probability $0< p< 1$ there exists a constant $c_{2,\rm Pra}=c_{2, \rm Pra}(p)>0$, such that for $G\in \Gnp$ with high probability
\[  \fPra (G) < c_{2, \rm Pra}\frac{n}{\log n}.
\]
\end{conj}
A matching lower bound   $ c_{1, \Kne} (n/(\log n))< \fPra (G)$ (with high probability) follows from~\eqref{eq:19Pra} and Corollary~\ref{co:1}.
We think the same order of magnitude holds for the case when $G$ is bipartite.

\begin{conj}\label{con:Pra2}
For a constant probability $0< p< 1$ there exists a constant $c_{3,\rm Pra}=c_{3, \rm Pra}(p)>0$, such that for $G\in \Gnnp$ with high probability
\[   \fPra (G) < c_{3, \rm Pra}\frac{n}{\log n}.
\]
\end{conj}

\subsection{Prague dimension and clique coverings of graphs}

The {\em chromatic index} $\theta'_0(\cC)$ of a clique cover $\cC:= \{ C_1, \dots, C_m\}$ of the graph $G$ is 
 the chromatic index of the hypergraph $\cC$, i.e.,
 $\theta'_0(\cC)$ is the smallest $k$ that one can decompose the clique cover into $k$ parts,
$\cC=\cC_1\cup \dots \cup \cC_k$ such that the members of each $\cC_i$ are pairwise (vertex)disjoint.
The minimum chromatic index among the clique covers of $G$ is denoted by $\theta'_0(G)$.
In other words, $E(G)$ can be covered by $k$ subgraphs with complete graph components.
Obviously, the thickness is a lower bound $ \theta_0(G) \leq   \theta'_0(G)$.
Here the left hand side is at most $O(n/(\log n))$ for almost all graphs by Theorem~\ref{th:thickness}. 
We think that the Frieze--Reed~\cite{FriezeReed} method can be applied to find the correct order of magnitude of $\theta'_0$, too.

\begin{conj}\label{con:betterFR}
For $p$ constant, $0< p<1$, there exists a constants $c_4=c_4(p)>0$ such that for $G\in \Gnp$ with high probability
\[     \theta'_0(G) <     c_4\frac{n}{\log n}.
\]
  \end{conj}

One can observe that (similarly as $\fKn$ and $\theta_0$ are related, see~\eqref{eq:thickness2})
 there is a remarkable simple connection between Prague dimension and $\theta'_0$.
\[   \theta'_0(G)\leq   \fPra(\overline{G})  \leq \theta'_0(G)+1.
 \]
So Conjectures~\ref{con:Pra} and~\ref{con:betterFR} are in fact equivalent, and Conjecture~\ref{con:betterFR}
   also implies Conjecture~\ref{con:Pra2}.

\section{Conclusion}

We have considered five graph functions  $\rmin(G)$, $\ravg(G)$, $\rmax(G)$, $\fKn(G)$, and $\fPra(G)$,  which are {\em hereditary} (monotone  for induced subgraphs) and two random graph models $\Gnp$ and $\Gnnp$. 
We gave an upper bound for the order of magnitude for eight of the possible ten problems, and we have also have 
 conjectures for the missing two upper bounds (Conjectures~\ref{con:Pra} and~\ref{con:Pra2}).
We also established matching lower bounds in seven cases, which also gave probably the best lower bound in two more cases (concerning $\fPra$).
All of these 19 estimates were $\Theta (n/(\log n))$. In the last case (in Corollary~\ref{co:probversion}) we have a weaker bound, 
 so it is natural to ask that

\begin{prob}
Is it true that for any fixed $0<p<1$ for $G\in \Gnp$ with high probability one has
$\Omega (n/(\log n))\leq \rmin (G)$?
  \end{prob}

Let us remark that if $G$ is a complement of a triangle-free graph then  
   the Kneser rank and Prague dimension is $\Delta(\overline{G})$ or $\Delta(\overline{G})+1$.
So  it can be $\Omega(n)$. 
For example, $\fKn(\overline{K_{1,n-1}})=n-1$. 
No such results are known for $\rmin$. 

\begin{prob}
What is the maximum of $\rmin(G)$ over the set of $n$-vertex graphs $G$? 
Is it true that $f_{\rho}(G)= o(n)$ for every $\rho\in \{ \min, \avg, \max\}$ and $G \in \Gn \cup \Gnn$?
\end{prob}




\end{document}